\documentclass[12pt]{amsart}
\usepackage{amssymb,amsmath,amsthm,latexsym}
\usepackage{mathrsfs,url}
\usepackage{a4wide}
\usepackage[usenames,dvipsnames]{color}

\usepackage[utf8]{inputenc}
\usepackage[T1]{fontenc}

\usepackage{mathrsfs,dsfont}
\usepackage[mathscr]{euscript}
\DeclareFontFamily{U}{mathx}{\hyphenchar\font45}
\DeclareFontShape{U}{mathx}{m}{n}{
      <5> <6> <7> <8> <9> <10>
      <10.95> <12> <14.4> <17.28> <20.74> <24.88>
      mathx10
      }{}
\DeclareSymbolFont{mathx}{U}{mathx}{m}{n}
\DeclareFontSubstitution{U}{mathx}{m}{n}
\DeclareMathAccent{\widecheck}{0}{mathx}{"71}
\DeclareMathAccent{\wideparen}{0}{mathx}{"75}

\newtheorem{theorem}{Theorem}[section]
\newtheorem{lemma}[theorem]{Lemma}
\newtheorem{corollary}[theorem]{Corollary}
\newtheorem{proposition}[theorem]{Proposition}
\theoremstyle{remark}

\newtheorem*{tha*}{\emph{\textbf{Theorem A}}}

\newtheorem*{thb*}{\emph{\textbf{Theorem B}}}

\newtheorem*{remark*}{Remark}
\theoremstyle{definition}

\newtheorem{problem}[theorem]{Problem}

\numberwithin{equation}{section}
\makeatother

\newcommand{\vertiii}[1]{{\left\vert\kern-0.25ex\left\vert\kern-0.25ex\left\vert #1 
    \right\vert\kern-0.25ex\right\vert\kern-0.25ex\right\vert}}

\newcounter{smallromans}

\newenvironment{romanenumerate}
{\begin{list}{{\normalfont\textrm{(\roman{smallromans})}}}%
  {\usecounter{smallromans}\setlength{\itemindent}{0cm}%
   \setlength{\leftmargin}{5.5ex}\setlength{\labelwidth}{5.5ex}%
   \setlength{\topsep}{.5ex}\setlength{\partopsep}{.5ex}%
   \setlength{\itemsep}{0.1ex}}}%
{\end{list}}

\newcounter{smallromansdash}

{\end{list}}

\newcounter{bigromans} 
  {\end{list}}

\begin{document}
\date{\today}
\title[Unital Banach algebras not isomorphic to Calkin algebras]{Unital Banach algebras not isomorphic to Calkin algebras of separable Banach spaces}

\author{Bence Horv\'ath}
\address{Institute of Mathematics, Czech Academy of Sciences, \v{Z}itn\'{a} 25, 115~67 Prague 1, Czech Republic}
\email{horvath@math.cas.cz}

\author{Tomasz Kania}
\address{Institute of Mathematics, Czech Academy of Sciences, \v{Z}itn\'{a} 25, 115~67 Prague 1, Czech Republic and Institute of Mathematics, Jagiellonian University, {\L}ojasiewicza 6, 30-348 Krak\'{o}w, Poland}
\date{\today}
\email{kania@math.cas.cz, tomasz.marcin.kania@gmail.com}

\subjclass[2010]{47L10 (primary) and 46B03 (secondary)} 
\thanks{The authors acknowledge with thanks support received from GA\v{C}R project 19-07129Y; RVO 67985840.}
\keywords{Calkin algebra, group algebra, separable Banach space, unital Banach algebra}

\begin{abstract}Recent developments in Banach space theory provided unexpected examples of unital Banach algebras that are isomorphic to Calkin algebras of Banach spaces, however no example of a unital Banach algebra that cannot be realised as a~Calkin algebra has been found so far. This naturally led to the question of possible limitations of such assignments. In the present note we provide examples of unital Banach algebras meeting the necessary density condition for being the Calkin algebra of a separable Banach space that are not isomorphic to Calkin algebras of such spaces, nonetheless. The examples may be found of the form $C(X)$ for a compact space $X$, $\ell_1(G)$ for some torsion-free Abelian group, and a~simple, unital AF $C^*$-algebra. Extensions to higher densities are also presented.\end{abstract}
\maketitle

\section{Introduction and the main result}
Let $E$ be a Banach space. Denote by $\mathscr{B}(E)$ the Banach algebra of all (bounded, linear) operators on $E$, by $\mathscr{F}(E)$ the subalgebra of all finite-rank operators, by $\mathscr{A}(E)$ the norm-closure of $\mathscr{F}(E)$ (the \emph{approximable operators}), and finally by $\mathscr{K}(E)$ the space of all compact operators on $E$. Each aforementioned subspace is also a two-sided ideal of $\mathscr{B}(E)$. When $E$ is a Hilbert space (or more generally, $E$ has the Approximation Property) $\mathscr{A}(E) = \mathscr{K}(E)$ but in general $\mathscr{K}(E)$ may be strictly bigger. As the two ideals may be different, it is not clear what is the right definition of a Calkin algebra on a Banach space. In the present note we fix $\mathscr{Q}$ to be one of the possible assignments:
\begin{equation}\label{calkin}E \mapsto \mathscr{B}(E) / \mathscr{A}(E)\;\;\text{ or } \;\; E \mapsto \mathscr{B}(E) / \mathscr{K}(E)\;\;\;\;(E\text{ is a Banach space}),\end{equation}
and having done so we call $\mathscr{Q}(E)$ \emph{the} Calkin algebra of $E$.\smallskip

Calkin algebras of classical (infinite-dimensional) Banach spaces tend to be very large in the sense that not only are they non-separable, but they do not even linearly inject into $\ell_\infty$ as Banach spaces.\smallskip %(a more precise statement of this fact may be found in Proposition~\ref{nonSC}, which is a far reaching generalisation of the fact that $\mathscr{Q}(\ell_2)$ is not representable on a~separable Hilbert space). \smallskip 

The construction of the Argyros--Haydon space (\cite{ArgyrosHaydon:2011}) whose Calkin algebra is one-di\-men\-sio\-nal (the space has a basis so both assignments from \eqref{calkin} yield the same algebra) opened a zoo of new examples of Calkin algebras that \emph{are} separable. Such examples include finite direct sums of full matrix algebras of arbitrary dimensions (\cite[p.~4]{KaniaLaus}), $C(X)$, the algebra of continuous functions on arbitrary countable compact space $X$ (\cite{motakisetal}), the algebra of upper triangular matrices (\cite{KaniaLaus}), $\ell_1(\mathbb N_0)$ with convolution (\cite[Chapter IV]{Tarbard:2013}), the bidual of the $p$\textsuperscript{th} James space ($p\in (1,\infty)$) with pointwise multiplication (\cite[Theorem III]{Motakispuglisi}), and more. \smallskip

It is no wonder that the question:
\begin{quote}\emph{Which Banach algebras are isomorphic to Calkin algebras (in either sense) of Banach spaces?}
\end{quote}
was raised by several researchers. We thank the referee for pointing out that this question was recorded in \cite[p.~134]{Tarbard:2013}. It was also asked during the Thematic Program on Abstract Harmonic Analysis, Banach and Operator Algebras held in the Fields Institute, Toronto, in 2014.\smallskip

Since $\mathscr{B}(E)$ is always unital, a necessary condition for a Banach algebra to be isomorphic to a Calkin algebra is to have a unit. Up to now, it is the \emph{only} necessary condition known, so potentially every unital Banach algebra is realisable as a Calkin algebra. When we restrict to the class of Banach spaces of fixed density $\lambda$ (for example, separable Banach spaces; $\lambda = \omega$) another trivial necessary condition is available: considered unital Banach algebras cannot have density bigger than $2^\lambda$ because the density of $\mathscr{B}(E)$ is bounded by $2^\lambda$ (see Lemma~\ref{density}), thus of every quotient thereof.\smallskip

The aim of this note is to demonstrate that in this restricted setting there exist Banach algebras satisfying the two necessary conditions with respect to a fixed density of Banach spaces and belonging to familiar classes of algebras, yet they cannot be realised as Calkin algebras for this class. The idea is based on the pigeon-hole principle: for a fixed cardinal $\lambda$ there are at most $2^{\lambda}$ isomorphism-types of Calkin algebras of density $2^{\lambda}$ (by Lemma~\ref{lem:operatoralgebra}), but there are precisely $2^{2^\lambda}$ isomorphism-types of Banach algebras of density $2^{\lambda}$ (see Lemma~\ref{lem:howmany}). Hence, we focus on listing natural classes from which \emph{almost all} algebras are not realisable as Calkin algebras.
%despite the fact that it is known that certain algebras from these classes are isomorphic to Calkin algebras.

\begin{theorem}
Let $\lambda$ be an infinite cardinal number. Then there exists a unital Banach algebra $\mathscr{A}$ of density at most $2^\lambda$ that is not isomorphic to a Calkin algebra of any Banach space of density $\lambda$. Moreover, $\mathscr{A}$ may take each of the following forms:
\begin{romanenumerate}
    \item\label{C1} $\mathscr{A} = C(X)$ for some compact space $X$, which may be taken either to be
    \begin{itemize}
        \item an ordinal,
        \item a scattered Fr\'echet--Urysohn space,
        \item connected Abelian group, or
        \item when $\lambda = \omega$, separable, extremally disconnected, and rigid (that is, the identity map is the only auto-homeomorphism).
    \end{itemize} 
    \item $\mathscr{A}=\mathscr{K}(E)^\#$, the unitisation of $\mathscr{K}(E)$ for some Banach space $E$ of the form $C(X)$.
    \item $\mathscr{A}$ is a simple AF $C^*$-algebra.
    \item $\mathscr{A} = \ell_1(G)$ for some torsion-free Abelian group.
%    \item \textcolor{green}{$\mathscr{A} = \ell_1(G)$ for some group $G$.}
\end{romanenumerate}
\end{theorem}
It is quite remarkable that, by \cite{motakisetal}, for every countable ordinal $\alpha$ the algebra $C[0,\alpha]$ is isomorphic to a Calkin algebra of a separable Banach space; scatteredness of $[0,\alpha]$ translates into purely algebraic properties of the algebra of continuous functions (generation by minimal idempotents), hence it follows that likely there is no hope for realisability of algebraically defined classes of algebras as Calkin algebras at once. Of course, the main result of this note does not prevent the found counterexamples to be Calkin algebras of Banach spaces of higher densities. Thus the question of whether they are realisable as Calkin algebras arises. Let us record the natural problem of realisability of certain classes of unital Banach algebras as Calkin algebras.
\begin{problem}
Let $\lambda$ be an infinite cardinal number. Identify classes of unital Banach algebras (of density at most $2^\lambda$) that are \emph{not} isomorphic to Calkin algebras of Banach spaces of density $\lambda$.
\end{problem}

\section{Preliminaries and auxiliary results} In this note we consider Banach spaces over a fixed field of real or complex numbers. Let $E$ be a Banach space. A closed subalgebra of $\mathscr{B}(E)$ containing $\mathscr{F}(E)$ is called a \emph{standard operator algebra over} $E$. Eidelheidt's theorem asserts that two Banach spaces $E$ and $F$ are isomorphic if and only if $\mathscr{B}(E)$ and $\mathscr{B}(F)$ are isomorphic as rings. We shall require the following extension of Eidelheidt's theorem to standard operator algebras (\cite[Corollary 3.2]{Chernoff:1973}; see also \cite{AsadiKhosravi:2006}).
\begin{lemma}\label{lem:operatoralgebra}Let $\mathscr{A}_1, \mathscr{A}_2$ be two standard operator algebras over Banach spaces $E$ and $F$, respectively. If $\mathscr{A}_1$ and $\mathscr{A}_2$ are isomorphic as rings, then the Banach spaces $E$ and $F$ are isomorphic.\end{lemma} 
Let $S$ be a semigroup written multiplicatively. The Banach space $\ell_1(S)$, the space of purely atomic measures on $S$, is naturally a Banach algebra with the convolution product, which is determined on the standard unit vector basis $(e_s)_{s\in S}$ of $\ell_1(S)$ by the rule $e_s \ast e_t = e_{s\cdot t}$ ($s,t\in S$). When $S$ is a monoid, $\ell_1(S)$ is unital.\smallskip

The \emph{density} of a topological space $X$ is the least cardinality of a dense subset of $X$, whereas the \emph{weight} of $X$ is the least cardinality of a topological base of $X$. When $X$ is compact, it follows from the Stone--Weierstra{\ss} theorem that the density of the Banach space $C(X)$ coincides with the weight of $X$. For an infinite cardinal number, we denote by $D(\lambda)$ the discrete space of cardinality $\lambda$ (whose underlying set is $\lambda$) and by $\lambda^+$ the cardinal successor of $\lambda$, the least cardinal number bigger than $\lambda$. A topological space $X$ is \emph{scattered}, whenever every non-empty subset of $X$ has an isolated point. A topological space $X$ is \emph{Fr\'echet--Urysohn}, whenever for every subset $A$ of $X$ the closure of $A$ coincides with the sequential closure.\smallskip

For the remainder of this section we fix an infinite cardinal number $\lambda$. \smallskip 

The next lemma is certainly widely known, however we include its proof for the reader's convenience.
\begin{lemma}\label{density}Let $E$ be a Banach space of density $\lambda$. Then $|\mathscr{B}(E)|\leqslant 2^\lambda$.
\end{lemma}
\begin{proof}
Every operator in $\mathscr{B}(E)$ is determined by its values on a fixed set of cardinality $\lambda$. Thus $|\mathscr{B}(E)| \leqslant |E|^\lambda  = (\lambda^\omega)^\lambda = \lambda^\lambda = 2^\lambda,$
where the first equality follows from the fact that each point is a limit of a sequence from a set of cardinality $\lambda$, whereas the final equality follows from Cantor's theorem as
$2^\lambda\leqslant\lambda^\lambda\leqslant\left(2^\lambda\right)^\lambda=2^{\lambda\cdot\lambda}=2^\lambda.$
\end{proof}

We start by recording the following statements.

\begin{proposition}\label{Bashkirov}
There exist
\begin{romanenumerate}
    \item $\lambda^+$ pairwise non-homeomorphic compactifications of $D(\lambda)$ that are scattered, Fr\'echet--Urysohn, and have cardinality $\lambda$;
    \item\label{explambda} $2^\lambda$ pairwise non-homeomorphic Fr\'echet--Urysohn compactifications of weight $\lambda$.
    \item $2^\lambda$ pairwise non-homeomorphic compact and connected Abelian groups of weight $\lambda$.
    \item $2^{\mathfrak c}$ pairwise non-homeomorphic extremally disconnected, separable compact spaces that are topologically rigid. 
\end{romanenumerate}
\end{proposition}
\begin{proof}
The first two clauses are due to Bashkirov (\cite[Theorem 1 \& Lemma 4]{Bashkirov:1980}). The penultimate clause is well-known and stronger results are available; see, \emph{e.g.}, \cite[Theorem 1.4]{OrsattiRodino:1986}). The final clause is due to Dow, Gubbi, and Szyma\'nski (\cite{Dowetal:1988}).\end{proof}

Rudin observed that when $X$ is a compact scattered space, then the dual space of $C(X)$ is isometric to $\ell_1(X)$ (\cite[Theorem 6]{Rudin:1957}). In particular, for every scattered compactification $X$ of $D(\lambda)$ that has cardinality $\lambda$ the dual space $C(X)^*\cong \ell_1(X)$ has density $\lambda$. Let $\alpha$ be an~ordinal number. Then the ordinal interval $[0,\alpha]$ is compact and scattered. Isomorphism types of the spaces $C[0,\alpha]$ are distinguished, for example, by the ordinal-valued Szlenk index, which is an isomorphic invariant of Banach spaces (\cite[Fact 2.38]{Hajeketal}). We do not intend to define the Szlenk index here, instead we refer to \cite[Section 2.4]{Hajeketal} for the definition and basic properties thereof. The Szlenk index was computed for countable $\alpha$ by Samuel \cite{Samuel:1984} and the computation was extended to arbitrary ordinals by Brooker \cite{Brooker:2013}. 

\begin{proposition}\label{prop:Szlenk}
Let $\alpha$ be an infinite ordinal number. Then ${\rm Sz}\, C[0,\alpha] = \omega^{\gamma+1}$, where $\gamma$ is the unique ordinal number such that $\omega^{\omega^\gamma}\leqslant \alpha < \omega^{\omega^{\gamma+1}}$.
\end{proposition}
\begin{corollary}\label{cor:manytypes}
For every cardinal number $\lambda$ there are precisely $\lambda$ Banach space isomorphism types of the spaces $C[0,\alpha]$ for ordinals $\alpha < \lambda$.
\end{corollary}

The subsequent lemma is certainly folk knowledge; we include it with a proof for the sake of completeness.
\begin{lemma}\label{lem:compact}
Let $E$ be a Banach space and let $\lambda$ be the density of $E^*$. Then $\mathscr{K}(E)$ has density $\lambda$.
\end{lemma}
\begin{proof}
There is an isometric isomorphism between $\mathscr{K}(E)$ and the space $\mathscr{K}^{{\rm w}^*}(E^*)$ of weak*-continuous compact operators on $E^*$. Indeed, by Schauder's theorem (see \textit{e.g.} \cite[Theorem~3.4.15]{Megginson}), $T\in \mathscr{K}(E)$ if and only if $T^*\in \mathscr{K}(E^*)$. Secondly, an operator is weak*- to weak*-continuous on $E^*$ if and only if it is an adjoint of some bounded linear operator on $E$ (see \textit{e.g.} \cite[Proposition~2.4.12]{Pedersen}). Consequently, the restriction of the adjoint operation
$\mathscr{K}(E) \to \mathscr{K}^{{\rm w}^*}(E^*); \; T \mapsto T^*$
is an isometric isomorphism. It is enough then to show that the space $\mathscr{K}^{{\rm w}^*}(E^*)$ of weak*-continuous compact operators on $E^*$ has density $\lambda$. 

Every operator $T\in \mathscr{B}(E^*)$ is uniquely determined by its restriction to $B_{E^*}$, which is compact in the weak*-topology by the Banach--Alaoglu theorem. Thus, for an operator $T\in \mathscr{K}^{{\rm w}^*}(E^*)$, $T|_{B_{E^*}}\in C(B_{E^*}, E^*)$ because by compactness, $T^*$ is weak*-to-norm continuous; here $C(B_{E^*}, E^*)$ denotes the Banach space of $E^*$-valued continuous functions on $B_{E^*}$  endowed with the relative weak*-topology. Now $C(B_{E^*}, E^*)$ is isometrically isomorphic to the injective tensor product of $C(B_{E^*})$ and $E^*$ (see \textit{e.g.} \cite[p.~50]{Ryan}). The space $C(B_{E^*})$ has density at most $\lambda$ because the weight of $B_{E^*}$ in the weak*-topology does not exceed the metric weight of $B_{E^*}$, which is $\lambda$. Thus, $C(B_{E^*}, E^*)$ has density $\lambda$. Consequently, $\mathscr{K}^{{\rm w}^*}(E^*)$ has density at most $\lambda$. However, $\mathscr{K}(E)$ contains isomorphic copies of $E^*$ so this density is precisely $\lambda$.\end{proof}

\begin{proposition}
For every cardinal number $\lambda$ there are precisely $\lambda^+$ Banach algebra isomorphism types of the algebras $\mathscr{K}(C[0,\alpha])$ for $\alpha \in [\lambda, \lambda^+)$. Moreover every above-listed space has density $\lambda$.
\end{proposition}
\begin{proof}By Corollary~\ref{cor:manytypes}, there are $\lambda^+$ isomorphism types of the spaces $C[0,\alpha]$, where $\alpha \in [\lambda, \lambda^+)$ is an ordinal. Each space $C[0,\alpha]$ has density $\lambda$, because the weight of $[0, \alpha]$ is $\lambda$. Since $[0,\alpha]$ is scattered, the dual space of $C[0,\alpha]$ is isometric to $\ell_1([0,\alpha])$, which has density $\lambda$ too. By Lemma~\ref{lem:compact}, the space $\mathscr{K}(C[0,\alpha])$ has density $\lambda$. Finally, for ordinals $\alpha, \beta \in [\lambda, \lambda^+)$, by Lemma~\ref{lem:operatoralgebra}, the following statements are equivalent:
\begin{itemize}
    \item $C[0,\alpha] \cong C[0,\beta]$ as Banach spaces,
    \item $\mathscr{K}(C[0,\alpha]) \cong \mathscr{K}(C[0,\beta])$ as rings,
\end{itemize}
which concludes the proof.\end{proof}

\begin{lemma}\label{lem:howmany}Let $\kappa$ be an infinite cardinal. Every unital Banach algebra of density $\kappa$ is a~quotient of $\ell_1(S_\kappa)$, where $S_\kappa$ is the free monoid on $\kappa$ generators. Consequently, there are precisely $2^\kappa$ pairwise non-isomorphic unital Banach algebras of density $\kappa$. \end{lemma}

\begin{proof} It follows from the Gelfand--Kolmogorov theorem and Proposition~\ref{Bashkirov}(ii), that there are at least $2^\kappa$ pairwise non-isomorphic algebras of the form $C(X)$ each having density $\kappa$. Thus, it is enough to show that there are at most $2^\kappa$ non-isomorphic unital Banach algebras in total.

Denote by $S_\kappa$ the free monoid on $\kappa$ generators. Let $A$ be a unital Banach algebra of density $\kappa$ and let $D_A$ be a fixed dense subset of the unit ball $B_A$ of $A$ that has cardinality $\kappa$. The unit ball of $A$ is is a monoid with the multiplication inherited from $A$. Let $\langle D_A\rangle$ be the monoid generated by $D_A$ in $B_A$. In particular, $\langle D_A\rangle$ has cardinality $\kappa$ too. Since $S_\kappa$ is free, we may fix a~surjective monoidal homomorphism $\theta_A\colon S_\kappa\to \langle D_A\rangle$. We extend the assignment $$\vartheta_A\colon e_t\mapsto \theta_A(t)\quad (t\in S_\kappa)$$
uniquely to a linear map $h_\theta\colon c_{00}(S_\kappa)\to A$ from the semigroup algebra $c_{00}(S_\kappa)$ of $S_\kappa$, for which we can take a further unique extension to the completion $\overline{h_\theta}\colon \ell_1(S_\kappa)\to A$. Since $D_A$ is dense in the unit ball of $A$ (and in the domain we consider the $\ell_1$-norm), $\overline{h_\theta}$ is a surjective map. Moreover, it is a (continuous) algebra homomorphism. Consequently, $A\cong \ell_1(S_\kappa) / \ker \overline{h_\theta}$ as Banach algebras. Hence every unital Banach algebra  of density $\kappa$ corresponds to a closed ideal of $\ell_1(S_\kappa)$, which in turn corresponds to a homomorphism from $S_\kappa$ to a set of cardinality $\kappa$, of which there are at most $\kappa^\kappa = 2^\kappa$ many.
\end{proof}
%\begin{proposition}\label{nonSC}
%Let $E$ be a Banach space with an unconditional Schauder decomposition. Then $c_0(\mathfrak{c})$ embeds into $\mathscr{Q}(E)$. In particular, there exists no linear injection from $\mathscr{Q}(E)$ into $\ell_\infty$.\end{proposition}
%\begin{proof} Without loss of generality we may assume that $E$ has countable unconditional Schauder decomposition $(E_n)_{n=1}^\infty$; let $(P_n)_{n=1}^\infty$ be the corresponding projections. For each $n$, pick a rank-one projection $Q_n$ with range contained in $E_n$. Thus, for any subset $N\subset \mathbb N$ it makes sense to define pointwise the operator $Q_N = \sum_{n\in N} P_nQ_nP_n = \sum_{n\in N} Q_nP_n$ and for any $N,M\subseteq \mathbb N$, we have $Q_NQ_M=Q_MQ_N$. When $N$ is infinite, $Q_N$ fails to be compact. Choose an almost disjoint family $\mathscr{D}$ of subsets $\mathbb N$ (any two members have finite intersection) that has cardinality $\mathfrak{c}$. Let $q\colon \mathscr{B}(E)\to \mathscr{Q}(E)$ be the canonical quotient map. Then the family $\{q(Q_N)\colon N\in \mathscr{D}\}$ comprises pairwise orthogonal  idempotents. Moreover, the closed linear span  of $\{q(Q_N)\colon N\in \mathscr{D}\}$ is isometric to $c_0(\mathscr{D})$, in which $\{q(Q_N)\colon N\in \mathscr{D}\}\cup\{0\}$ is a non-separable weakly compact subset. Consequently, there is no linear injection from $\mathscr{Q}(E)\to \ell_\infty$ as every weakly compact subset in the latter space is separable (see also \cite[Fact 4.10]{Hajeketal}).\end{proof}
\section{Proof of Theorem 1.1}
We are now ready to prove the main result of the paper.

\begin{proof}[Proof of Theorem 1.1] Let $\lambda$ be an infinite cardinal number. By Lemma~\ref{density}, the Calkin algebra of a Banach space of density $\lambda$ may have itself density at most $\kappa = 2^\lambda$. It is thus enough to justify that there are more, in terms of cardinality, unital Banach algebras from a given class than Banach spaces of density $\lambda$ of which there are at most $\kappa$ many. Let us justify this case by case.

\begin{romanenumerate}
\item It follows from Corollary~\ref{cor:manytypes} applied to $\kappa$ that there are $\kappa^+$ many ordinals $\alpha \in [\kappa, \kappa^+)$ such the spaces $C[0,\alpha]$ are pairwise non-isomorphic as Banach spaces; every space of this form has density $\kappa$. The conclusion then follows.

 Proposition~\ref{Bashkirov}(i) assures that there are $\kappa^+$ scattered compact Fr\'echet--Urysohn spaces of weight $\kappa$, which yield $C(X)$-spaces of density $\kappa$. Proposition~\ref{Bashkirov}(iii--iv) establishes the remaining cases. 
 
 (We point out that separability of the extremely disconnected compact spaces mentioned in the last clause of (i) implies that the corresponding spaces of continuous functions are isomorphic to $\ell_\infty$, hence they have density $\mathfrak{c}$. Indeed, every separable compact space $X$ is a continuous image of $\beta \mathbb N$, the \v{C}ech--Stone compactification of the discrete space of natural numbers, so $C(X)$ is then naturally a subspace of the space $C(\beta \mathbb N)\cong \ell_\infty$. When $X$ is extremely disconnected and infinite, $C(X)$ is then injective, so complemented in $\ell_\infty$, and infinite-dimensional, hence isomorphic to $\ell_\infty$ by \cite{Lindenstrauss:1967}.)
\item By Corollary~\ref{cor:manytypes}, there are $\kappa^+$ pairwise non-isomorphic Banach spaces of the form $C[0,\alpha]$ for ordinals $\alpha<\kappa^+$, whose dual spaces in this case have density $\kappa$. Consequently, Lemma~\ref{lem:compact} applies so that the spaces of compact operators acting thereon have density $\kappa$ too. By Lemma~\ref{lem:operatoralgebra}, we conclude that there are $\kappa^+$ pairwise non-isomorphic algebras of the form $\mathscr{K}(C[0,\alpha])$, where $\alpha \in[\kappa, \kappa^+)$ is an ordinal. The conclusion extends to the unitisations of the respective algebras.
\item This follows directly from \cite{FarahKatsura:2015}, where it is proved that for every uncountable cardinal $\kappa$ there are $2^\kappa$ non-isomorphic simple AF algebras of density character $\kappa$ and the fact that two unital $C^*$-algebras are isomorphic as $C^*$-algebras if and only if they are isomorphic as Banach algebras (see \cite[Theorem~B]{Gardner}).
\item Proposition~\ref{Bashkirov}(iii) yields $2^\kappa$ pairwise non-homeomorphic compact and connected Abelian groups of weight $\kappa$. By the standard Pontryagin duality (\cite[Theorems 24.25 and 24.14]{HewittRoss}), the groups dual to compact and connected Abelian groups of weight $\kappa$ are Abelian, have cardinality $\kappa$, and are torsion-free. This yields $2^\kappa$ pairwise non-isomorphic algebras of the form $\ell_1(G)$ with $G$ having the announced properties. Indeed, a Banach-algebra isomorphism between two unital Banach algebras induces a homeomorphism between their spaces of maximal ideals, which in this case are precisely the said compact groups (\cite[Theorem 23.15]{HewittRoss}).
\end{romanenumerate}\end{proof}
\subsection*{Acknowledgements} The authors are grateful to the anonymous referee for the careful reading of the manuscript, and for the very helpful suggestions.

\end{document}